\documentclass[11pt]{article}
\usepackage[margin=1in]{geometry}
\usepackage{mathtools}
\usepackage{chngcntr}
\counterwithout{equation}{section}
\usepackage{amsthm,amsfonts,amsmath,amssymb,epsfig,color,float,verbatim}
\usepackage{bbm}
\usepackage[colorlinks=true, linkcolor=blue, citecolor=blue, urlcolor=blue]{hyperref}
\usepackage[capitalize]{cleveref}
\usepackage{fancyvrb}

\usepackage[mathscr]{euscript}
\newcommand*{\myov}[1]{\overbracket[0.65pt][-1pt]{#1}}

\usepackage{xcolor}

\newcommand{\pred}[1]{\chr[#1]}
\newcommand{\nrm}[1]{\left\Vert #1 \right\Vert}

\newcommand{\norm}[1]{\Vert #1 \Vert}
\newcommand{\E}{\mathop{\mathbb{E}}}
\newcommand{\W}{\mathop{\mathbb{W}}}
\renewcommand{\P}{\mathop{\mathbb{P}}}
\newcommand{\R}{\mathbb{R}}

\newcommand{\paren}[1]{\left( #1 \right)}
\newcommand{\sqprn}[1]{\left[ #1 \right]}

\newcommand{\set}[1]{\left\{ #1 \right\}}

\newcommand{\abs}[1]{\left| #1 \right|}

\newcommand{\beq}{\begin{eqnarray*}}
\newcommand{\eeq}{\end{eqnarray*}}
\newcommand{\beqn}{\begin{eqnarray}}
\newcommand{\eeqn}{\end{eqnarray}}
\newcommand{\ben}{\begin{enumerate}}
\newcommand{\een}{\end{enumerate}}
\newcommand{\bit}{\begin{itemize}}
\newcommand{\eit}{\end{itemize}}
\providecommand{\hide}[1]{}

\newcommand{\eps}{\varepsilon}

\newcommand{\floor}[1]{\ensuremath{\left\lfloor#1\right\rfloor}}

\newcommand{\chr}{\boldsymbol{{1}}} 

\newcommand{\X}{\Omega}

\newcommand{\Lip}{\operatorname{Lip}}
\newcommand{\BV}{\operatorname{BV}}

\newcommand{\lip}[1]{\nrm{#1}_{\Lip}}

\newcommand{\fat}{\operatorname{fat}}
\renewcommand{\d}{\mathrm{d}}

\newcommand{\barLs}{{\myov{\Lip}}^{\sf{s}}}
\newcommand{\barLw}{{\myov{\Lip}}^{\sf{w}}}

\newcommand{\La}{\Lambda}

\newcommand{\wa}[1]{\norm{#1}_{\textrm{{\tiny \textup{W}}}}}
\newcommand{\sa}[1]{\norm{#1}_{\textrm{{\tiny \textup{S}}}}}

\newcommand{\Fw}{F_{\textrm{{\tiny \textup{W}}}}}
\newcommand{\Fs}{F_{\textrm{{\tiny \textup{S}}}}}

\renewcommand{\phi}{\varphi}

\begin{document}

\markboth{Ariel Elperin and Aryeh Kontorovich}{Bounded variation separates weak and strong average Lipschitz}

\newtheorem{theorem}{Theorem}[section]
\newtheorem{proposition}{Proposition}
\newtheorem{lemma}{Lemma}
\newtheorem{corollary}{Corollary}
\newtheorem{definition}{Definition}
\newtheorem{remark}{Remark}
\newtheorem{assumption}{Assumption}
\newtheorem{claim}[theorem]{Claim}
\newtheorem{fact}{Fact}[section]

\title{Bounded variation separates weak and strong average Lipschitz}

\author{%
  Ariel Elperin \\
  \texttt{ariel.elperin@gmail.com}
\and
  Aryeh Kontorovich \\
  \texttt{karyeh@cs.bgu.ac.il}
  }
  
\maketitle

\begin{abstract}

We closely examine 
a notion of average smoothness recently introduced by Ashlagi et al. (JMLR, 2024).
The latter defined a {\em weak} and {\em strong}
average-Lipschitz seminorm for real-valued functions on general metric spaces.
Specializing to the standard metric on the real line, we compare these
notions to bounded variation (BV) and discover that the weak notion is strictly
weaker than BV while the strong notion strictly stronger.
Along the way, we discover that the weak average smooth class
is also considerably larger in a certain combinatorial sense,
made precise by the fat-shattering dimension.

\end{abstract}

\section{Introduction}

A function $f:[0,1]\to\R$ is $L$-Lipschitz if $\abs{f(x)-f(x')}\le L\abs{x-x'}$
for all $x,x'\in[0,1]$,
and $\lip{f}$ is the smallest $L$ for which this holds.
If 
$f$ 
has an
integrable derivative, its {\em variation}
$V(f)$ is given by
$V(f)=\int_0^1\abs{f'(x)}\d x$ (the more general definition is given in \eqref{eq:bvdef}).
Since $\abs{f'(x)}\le \lip{f}$, we have the obvious relation
$V(f)\le 
\lip{f}
$. No reverse inequality is possible: since for monotone $f$,
we have $V(f)=|f(0)-f(1)|$
\cite{MR3156940}, a function 
whose value
increases from $0$ to $\eps$ with a sharp ``jump'' in the middle can have $L$ arbitrarily large and $V$ arbitrarily small.

Motivated by questions in machine learning and statistics,
Ashlagi et al.
\cite{AshlagiGK24} introduced 
two notions of ``average Lipschitz'' in general metric probability spaces: a weak one and a strong one.\footnote{Follow-up works extended these results
to average H\"older smoothness
\cite{HannekeKK24,KornowskiHK23}.
}
For the special case of the metric space $\Omega=[0,1]$ equipped with the standard metric $\rho(x,x')=|x-x'|$ and the uniform distribution $U$,
their definitions are as follows.
Both notions rely on the {\em local slope} of 
$f:[0,1]\to\R$ at a point $x$, defined (and denoted)
as follows:
\beqn
\label{eq:lamdef}
\Lambda_f(x) & =& 
\sup\limits_{
x'\in[0,1]\setminus\set{x}
}\frac{|f(x)-f(x')|}{|x-x'|},
\qquad x\in[0,1].
\eeqn
The {\em strong} and 
{\em weak} average smoothness of $f$
are defined, respectively, as
\beq
\label{eq:blas}
\sa{f}
&=& \E
\sqprn{\La_f(X)}
,\\
\label{eq:blaw}
\wa{f}
&=& \W
\sqprn{\La_f(X)}
=
\sup_{t>0}tU\paren{\set{x\in\X:\La_f(x)\ge t}},
\eeq
where $X$ is a random variable distributed according to $U$ on $[0,1]$, $\E$ is the usual expectation, and $\W$
is the {\em weak $L_1$ norm} 
of the random variable $Z$:
\beq
\label{eq:weakmean}
\W[Z]
=
\sup_{t>0}t\P(|Z|\ge t).
\eeq
Both 
$\sa{\cdot}$
and
$\wa{\cdot}$
satisfy the homogeneity axiom of seminorms (meaning that $\nrm{\alpha f}=
|\alpha|\cdot\nrm{f}$),
and 
$\sa{\cdot}$ additionally satisfies the triangle inequality
and hence is a true seminorm. 
The {\em weak $L_1$} norm
satisfies 
the weaker inequality $
\W[X+Y]\le
2(\W[X]+\W[Y])$ \cite{weakL1},
which $\wa{\cdot}$ also inherits.

We now recall
the definition of the variation of $f:[a,b]\to\R$:
\beqn
\label{eq:bvdef}
V_a^b(f)=\sup\limits_{
a
= x_0<x_1<x_2<\ldots<x_n\le 
b
}\sum\limits_{i=1}^n |f(x_i)-f(x_{i-1})|
\label{eq:Vdef}
\eeqn
(when $a=0$ and $b=1$ we omit these),
as well as the Liptschiz and Bounded Variation function classes:
\beq
\label{eq:lipdef}
\Lip &=& \set{f:[0,1]\to\R;~ \lip{f}<\infty}, \\
\label{eq:BV}
\BV &=& \set{f:[0,1]\to\R;~ V(f)<\infty}.
\eeq
The discussion above implies the (well known) strict containment
\beqn
\label{eq:lipbv}
\Lip\subsetneq\BV
.
\eeqn
In addition, we define the strong and weak
average smoothness classes
\beq
\label{eq:slipdef}
\barLs &=& \set{f:[0,1]\to\R;~ \sa{f}<\infty},\\
\label{eq:wlipdef}
\barLw &=& \set{f:[0,1]\to\R;~ \wa{f}<\infty}.
\eeq
By Markov's inequality and the fact that the expectation is bounded by the supremum, we have
\beq
\wa{f}\le\sa{f}\le\sup\limits_{x\in\Omega}\Lambda_f(x)=\lip{f}
\eeq
whence
\beqn
\label{eq:lipsw}
\Lip\subseteq 
\barLs \subseteq \barLw
;
\eeqn
all of these containments were shown to be strict
in \cite{AshlagiGK24}.
The containments in \eqref{eq:lipbv} and \eqref{eq:lipsw}
leave open the relation between $\BV$ and $
\barLs , \barLw
$, which 
we resolve in this work:
\begin{theorem}
\label{thm:main}
$\barLs\subsetneq \BV\subsetneq \barLw$.
\end{theorem}
We also provide a quantitative, finitary relation between these clases:
\begin{theorem}\label{thm:containment}
For any $f:[0,1]\to\R$, we have
$
\frac{1}{2}\wa{f}\le V(f)\le \sa{f}.
$
\end{theorem}

Finally, we recall the definition of the {\em fat-shattering dimension}, a combinatorial complexity measure
of function classes
of central importance in statistics, empirical processes, and machine learning
\cite{alon97scalesensitive,DBLP:journals/jcss/BartlettL98}.
Let $F$ be a collection of functions mapping $[0,1]$ to $\R$.
For $\gamma>0$,
a set $S=\set{x_1,\ldots,x_m}\subset[0,1]$
is said to be $\gamma$-shattered by
$F$
if 
\beqn
\label{eq:gamma-shatter}
\sup_{r\in\R^m}
\;
\min_{y\in\set{-1,1}^m}
\;
\sup_{f\in F}
\;
\min_{i\in[m]}
\;
y_i(f(x_i)-r_i)\ge \gamma.
\eeqn
The $\gamma$-fat-shattering dimension, denoted by $\fat_\gamma(F)$,
is the size of the largest $\gamma$-shattered set (possibly $\infty$).
It is known \cite{MR1741038} that for 
\( F = \{ f : [0,1] \to \mathbb{R} \mid V(f) \le L \} \), we have 
\( \fat_\gamma(F) = 1 + \left\lfloor \frac{L}{2\gamma} \right\rfloor \).
This same bound holds for 
\( F = \{ f : [0,1] \to \mathbb{R} \mid \Lip{f} \le L \} \).

Although the strong smoothness class has the same combinatorial complexity as the BV and Lipschitz classes,
for
weak average smoothness 
this quantity
turns out to to be considerably greater:
\begin{theorem}
\label{thm:averages-fat}
For $L>0$,
let $\Fw=
\set{
f:[0,1]\to\mathbb{R};
\wa{f}\le L
}
$ and $\Fs=\set{
f:[0,1]\to\mathbb{R};
\sa{f}\le L
}$.
Then:

\begin{enumerate}
    \item $\fat_\gamma(\Fw)=\infty$ whenever $\gamma\le\frac{L}{6}$
    \item $\fat_\gamma(\Fs)=1+\floor{\frac{L}{2\gamma}}$
    for $\gamma>0$.
\end{enumerate}
\end{theorem}

\paragraph{Notation.}
We write $[n]:=\set{1,\ldots,n}$
and use $m(\cdot)$ to denote the Lebesgue measure
(length) of sets in $\R$.

\section{Proofs}

We begin with a variant of the standard covering lemma.
\begin{lemma}\label{cover}
For any sequence $s_1,\ldots,s_n$ of closed segments in $\R$,
there is a subsequence indexed by $I\subseteq [n]$ such that for all distinct $i,j\in I$ we have $s_i\cap s_j=\emptyset$ and $\sum\limits_{i\in I}m(s_i)\ge\frac{1}{2}m\left(\bigcup\limits_{i=1}^n s_i\right)$.
\end{lemma}
\begin{proof}
We proceed by
induction on $n$. 
Let $G=([n],E)$ denote the intersection graph of the $s_i$:
the vertices correspond to the segments and $(i,j)\in E$ iff
$s_i\cap s_j\neq\emptyset$.

Suppose that $G$ contains a cycle, and let $s_1=[a_1,b_1],\ldots,s_k=[a_k,b_k]$ be the segments in the cycle sorted by their right endpoint. Since $s_1\cap s_k\neq\emptyset$, we have $a_k\le b_1$. If $a_{k-1}\ge a_1$ then $s_{k-1}\subseteq s_1\cup s_k$. Otherwise, $a_{k-1}<a_1$ and $s_1\subseteq s_{k-1}$. Either way, we have found a segment that is completely covered by the other vertices of $G$. After removing it we obtain $I\subseteq[n]$ of size $n-1$ with $\bigcup\limits_{i\in I}s_i=\bigcup\limits_{i=1}^n s_i$, so applying the inductive hypothesis on the segments in $I$ yields the desired result. If $G$ does not contain a cycle, then $G=A\cup B$ is bipartite, where $A, B\subseteq[n]$ are disjoint and nonempty.
Clearly $m\left(\bigcup\limits_{i=1}^n s_i\right)\le m\left(\bigcup\limits_{i\in A}s_i\right)+m\left(\bigcup\limits_{i\in B}s_i\right)$, and thus $\max\left(m\left(\bigcup\limits_{i\in A}s_i\right) ,m\left(\bigcup\limits_{i\in B}s_i\right)\right)\ge \frac {1}{2}m\left(\bigcup\limits_{i=1}^n s_i\right)$, so taking either $I=A$ or $I=B$ (which is possible since the segments inside each part are disjoint) yields the desired result.
\end{proof}

Next, we reduce the proof of Theorem~\ref{thm:containment}
to the case of right-continuous monotone functions.
\begin{lemma}\label{right_continuity}
If for every right-continuous monotone function $f:[0,1]\rightarrow\mathbb{R}$ we have $\wa{f}\le 2V(f)$ then the bound holds for all $f:[0,1]\rightarrow\mathbb{R}$. 
Furthermore, both inequalities are tight.
\end{lemma}
\begin{proof}
We begin by observing that we can restrict our attention to monotone functions, since $T_f(x)=V_f([0,x])$ is monotone and has the same variation as $f$, but $
\Lambda_{T_f}(x)=
\sup\limits_{x'\neq x}\frac{|T_f(x)-T_f(x')|}{|x-x'|}
\ge
\sup\limits_{x'\neq x}\frac{|f(x)-f(x')|}{|x-x'|}$,
which means $\wa{T_f}\ge \wa{f}$. 

Thus $\wa{T_f}\le 2V(T_f)$ immediately implies $\wa{f}\le\wa{T_f}\le 2V(T_f)=2V(f)$. If $f$ is monotone, it can only have jump discontinuities. Let $I\subset[0,1]$ denote the set of right discontinuities of $f$. Note that since $f$ is monotone, $I$ is at most countable. Define the modified version of $f$ to be:
\begin{align}
\widetilde{f}(x)=\begin{cases} f(x) & x\notin I \\ \lim_{\eps\rightarrow 0^+}f(x+\eps) & x\in I
.
\end{cases}
\end{align}
Note that
$\widetilde{f}$ is monotone and right-continuous. It is not hard to see that if $0\notin I$ then $V(\widetilde{f})=V(f)$ and $\Lambda_{\widetilde{f}}(x)=\Lambda_f(x)$ for all $x\notin I$, which implies $\wa{f}=\wa{\widetilde{f}}$ and allows us to restrict our discussion to right-continuous functions. If $0\in I$, then we can extend the domain of $\widetilde{f}$ to $[-\eps,1]$ for all $\eps>0$, where $\widetilde{f}(x)=f(0)$ for all $x<0$. Denote the extended function by $\widetilde{f}_\eps$, then since $\wa{f}=\lim_{\eps\rightarrow 0}\wa{\widetilde{f}_\eps}$ and $V(\widetilde{f}_\eps)=V(f)$ for all $\eps>0$, we can conclude that
\beq
\wa{f}=\lim_{\eps\rightarrow 0}\wa{\widetilde{f}_\eps}\le \lim_{\eps\rightarrow 0}2V(\widetilde{f}_\eps)=2V(f).
\eeq
\end{proof}

\subsection{Proof of Theorem~\ref{thm:containment}}

We first show that $\wa{f}\le 2V(f)$. 
We may assume without loss of generality that $V(f)<\infty$.
Since $f$ is of bounded variation, the function $T_f(x)=V_f([0,x])$ is well-defined for $x>0$. By Lemma \ref{right_continuity}, we may assume without loss of generality that $f$ is right-continuous.
Thus,
$T_f:[0,1]\rightarrow\mathbb{R}$ is monotone and right-continuous and thus induces a Lebesgue–Stieltjes measure on $[0,1]$, which we denote by $\mu_f$. We now define the {\em maximal} function $M_{f}:[0,1]\rightarrow\mathbb{R}$ as follows:
\begin{align}
M_f(x)=\sup\limits_{r_1,r_2>0}\frac{\mu_f([x-r_1,x+r_2])}{r_1+r_2}=\sup\limits_{r_1,r_2>0}\frac{V_f([x-r_1,x+r_2])}{r_1+r_2},
\end{align}
where
the segments $[a,x],[x,b]$ are taken to be $[0,x],[x,1]$, respectively, whenever $a<0$ or $b>1$. 
A standard argument shows that
$M_f^{-1}\left((t,\infty)\right)$ is open, whence $M_f$ is measurable.

We now observe that $M_f\ge \Lambda_f$ everywhere in $[0,1]$. Indeed, if $x'>x$ then 
$M_f(x)\ge\frac{V_f([x-\eps,x'])}{\eps+(x'-x)}\ge\frac{|f(x')-f(x)|}{x'-x+\eps}$
holds
for $ \eps>0$,
and hence
$M_f(x)\ge\sup\limits_{x'>x} \frac{|f(x')-f(x)|}{x'-x}$. The case of $x'<x$ is completely analogous,
whence 
$M_f(x)\ge \sup\limits_{x'\neq x}\frac{|f(x)-f(x')|}{|x-x'|}=\Lambda_f(x)$. 
For $X$ uniformly distributed over $[0,1]$ we have $\P\left(\Lambda_f(X)\ge t\right)\le 
\P\left(M_f(X)\ge t\right)$ and showing $\wa{f}\le 2V(f)$ reduces to bounding the latter probability by ${2V(f)}/{t}$.

We now closely follow the proof of Theorem 7.4 in \cite{Rudin}, and bound $m\left(M_f^{-1}\left((t,\infty)\right)\right)$ by bounding $m(K)$ for arbitrary compact $K\subseteq M_f^{-1}\left((t,\infty)\right)$. For $x\in K\subseteq M_f^{-1}\left((t,\infty)\right)$, denote by $r_1(x),r_2(x)$ some lengths such that $\frac{\mu_f\left([x-r_1(x),x+r_2(x)]\right)}{r_1(x)+r_2(x)}\ge t$. Denote by $S_x$ the open interval $(x-r_1(x),x+r_2(x))$, then clearly $K\subseteq\bigcup\limits_{x\in K}S_x$. Since $K$ is compact, a finite cover by intervals $S_{x_1},..., S_{x_n}$ exists. By Lemma \ref{cover} there exists $I\subseteq [n]$ such that for all distinct $i,j\in I$ we have $S_{x_i}\cap S_{x_j}=\emptyset$ and $\sum\limits_{i\in I}m(S_{x_i})\ge \frac{1}{2}m\left(\bigcup\limits_{i=1}^n S_{x_i}\right)$. Finally, by definition of the $S_x$'s, for each $i\in[n]$ it holds that $m(S_{x_i})\le \frac{\mu_f(S_{x_i})}{t}$. We can now write
\begin{align}
m(K)\le m\left(\bigcup\limits_{i=1}^n S_{x_i}\right)\le 2\sum\limits_{i\in I}m(S_{x_i})\le\frac{2}{t}\sum\limits_{i\in I}\mu_f(S_{x_i})\le\frac{2}{t}\mu_f([0,1]),
\end{align}
where the last inequality holds since the intervals in $I$ are disjoint. 
Since $\mu_f([0,1])=V(f)$,
it immediately follows that
$\wa{f}\le 2V(f)$.

It remains to show that $V(f)\le \sa{f}$. Let us denote by $P_n$ the partition $0\le x_1< x_2<...< x_n\le 1$ of $[0,1]$, and let $V(P_n)=\sum\limits_{i=1}^{n-1}|f(x_{i+1})-f(x_i)|$ denote the variation of $f$ relative to $P_n$. It suffices to show that for any such partition $P_n$ we have $\sa{f}\ge V(P_n)$.
Now
\begin{align}
\label{strong}
\sa{f}=\mathbb{E}\left[\Lambda_f(X)\right]\ge\sum\limits_{i=1}^{n-1}|x_{i+1}-x_i|\mathbb{E}\big[\Lambda_f(X)|X\in [x_i,x_{i+1}]\big].
\end{align}

Note that for all $x\in[x_i,x_{i+1}]$ we have:

\begin{align}
\Lambda_f(x)\ge\max\left(\frac{|f(x)-f(x_i)|}{x-x_i},\frac{|f(x_{i+1})-f(x)|}{x_{i+1}-x}\right)\ge \frac{|f(x_{i+1})-f(x_i)|}{x_{i+1}-x_i}
.
\end{align}
Applying this to \eqref{strong} yields
$$
\mathbb{E}\left[\Lambda_f\right]\ge\sum\limits_{i=1}^{n-1}|x_{i+1}-x_i|\frac{|f(x_{i+1})-f(x_i)|}{x_{i+1}-x_i}= V(P_n) 
.
$$

This concludes the proof

Finally, the tightness of the first claimed inequality
is witnessed by the step function $f(x)=\pred{x>1/2}$,
and of the second inequality by $f(x)=x$. 

$\square$

\subsection{Proof of Theorem~\ref{thm:main}}
The claimed containments are immediate from Theorem~\ref{thm:containment};
only the separations remain to be shown.
The first of these is obvious: the step function 
has bounded variation but infinite strong average
\cite[Appendix I]{AshlagiGK24}.
We proceed with the second separation:
\begin{lemma}\label{strict}
There exists an  $f:[0,1]\rightarrow [0,1]$ such that $V(f)=\infty$ but $\wa{f}\le 2$.
\end{lemma}
\begin{proof}
Let $f:[0,1]\rightarrow[0,1]$ be the piecewise linear function 
defined on $x_n=\frac1n$, $n\ge1$, by

$$f\left(\frac{1}{n}\right)=\sum\limits_{k=1}^n\frac{(-1)^{k+1}}{k}$$ 
and extended to $[0,1]$ by linear interpolation.

Clearly $V(f)=\sum\limits_{n=1}^\infty \frac{1}{n}=\infty$.
To bound 
$\wa{f}$,
note that any
$x,x'\in[0,1]$ 
witnessing
$\frac{|f(x)-f(x')|}{|x-x'|}\ge t$ 
also verify
$|x-x'|\le\frac{1}{t}$. 
Let $I_n$ denote the interval $\left[\frac{1}{n+1},\frac{1}{n}\right]$. 
If $\Lambda_f(x)\ge t$, then there is an $x'$ such that $\frac{|f(x)-f(x')|}{|x-x'|}\ge t$. 
Now either $x$ or $x'$ lies in $I_n$ for $n\ge t$. 
If $x\in I_n$ with $n\ge t$ then $x\le\frac{1}{t}$. If however $x'\in I_n$ for $n\ge t$ then $x'\le\frac{1}{t}$ and since $|x-x'|\le\frac{1}{t}$ we have $x\le\frac{2}{t}$. We conclude that $\Lambda_f(x)\ge t$ implies $x\le\frac{2}{t}$ and hence $\P\left(\Lambda_f(x)\ge t\right)\le\frac{2}{t}$;
this proves the claim.
\end{proof}

\paragraph{Remark.}Another function with this property is
$x\sin\frac{1}{x}$.

\subsection{Proof of Theorem~\ref{thm:averages-fat}}

\subsubsection{Proof that $\fat_\gamma(\Fw)=\infty$ whenever $\gamma\le\frac{L}{6}$}

Consider the partition of $[0,1]$ into segments $I_n=\left[x_{n+1},x_n\right]$ where $x_n=2^{-n}$. We define $f(x_n)=(-1)^n\gamma$. This specifies $f$ on all endpoints of $I_n$. For $x\in(x_{n+1},x_n)$ we define $f(x)=(-1)^n\frac{\gamma}{|I_n|}(x-x_{n+1})+(-1)^n\frac{\gamma}{|I_n|}(x-x_n)$, i.e. $f$ is piecewise linear with slope $(-1)^n4\gamma 2^n$ in $I_n$. Similarly 
to Lemma~\ref{strict}, 
if $\frac{|f(x)-f(x')|}{|x-x'|}\ge t$ then $|x-x'|\le\frac{2\gamma}{t}$. Now suppose $\Lambda_f(x)\ge t$, i.e. there exists $x'$ with $\frac{|f(x)-f(x')|}{|x-x'|}\ge t$. This implies that either $x$ or $x'$ lies in $I_n$ for some $n\ge\log{\frac{t}{4\gamma}}$ (the slope of the line connecting $x$ and $x'$ lies between the slopes of the segments containing $x,x'$). If $x\in I_n$ for some $n\ge\log\frac{t}{4\gamma}$, then $x\le x_n\le\frac{4\gamma}{t}$. If however $x'\in I_n$ for some $n\ge \log\frac{t}{4\gamma}$ then $x'\le \frac{4\gamma}{t}$ and since $|x-x'|\le\frac{2\gamma}{t}$ we have $x\le \frac{6\gamma}{t}$. Since $\Lambda_f(x)\ge t$ implies $x\le\frac{6\gamma}{t}$ we can conclude that $\wa{f}\le 6\gamma$. An immediate corollary is that $\barLw_L$ $\gamma$-shatters the infinite set $\{x_n\}_{n=1}^\infty$ for $\gamma\le\frac{L}{6}$ 
--- which is even stronger than having arbitrarily large $\gamma$ shattered sets.
$\square$
\subsubsection{Proof that 
$\fat_\gamma(\Fs)=1+\floor{\frac{L}{2\gamma}}$
    for $\gamma>0$}

The upper bound follows immediately from 
$V(f)\le 
\lip{f}
$. 
For the lower bound, take a $2\gamma/L$ packing $\set{x_1,\ldots,x_{1+\floor{\frac{L}{2\gamma}}}}$
of $[0,1]$. For labeling $y_i\in\set{-1,1}^n$ consider the linear interpolation of $\left(x_i, y_i\gamma\right)$, and observe that the interpolation $f$ satisfies $\Lambda_f(x)\le \frac{2\gamma}{2\gamma/L}=L$ everywhere.
$\square$


\end{document}